\documentclass[11pt,leqno]{amsart}
\usepackage{verbatim,amssymb,amsfonts,latexsym,amsmath,amsthm,bbm,nicefrac,xspace,enumerate}

\newtheorem{theorem}{Theorem}[section]
\newtheorem{lemma}[theorem]{Lemma}
\newtheorem{corollary}[theorem]{Corollary}

\newtheorem*{claim}{Claim}
\theoremstyle{definition}

\theoremstyle{remark}

\newcommand{\N}{\mathbb{N}}
\newcommand{\ben}{\mathbb{N}}

\newcommand{\Fix}{\hbox{\rm Fix}}
\newcommand{\nhat}[1]{\{1,2,\ldots,#1\}}

\newcommand{\pf}{{\mathcal P}_f}

\newcommand{\B}{\mathcal{B}}

\newcommand{\F}{\mathcal{F}}
\newcommand{\G}{\mathcal{G}}
\renewcommand{\H}{\mathcal{H}}

\newcommand{\RR}{\mathcal{R}}
\newcommand{\U}{\mathcal{U}}

\newcommand{\explicitSet}[1]{\left\lbrace #1 \right\rbrace}
\newcommand{\brackets}[1]{\left\langle #1 \right\rangle}
\newcommand{\set}[2]{\explicitSet{#1 \colon #2}}
\newcommand{\seq}[2]{\brackets{#1 \colon #2}}

\renewcommand{\a}{\alpha}
\renewcommand{\b}{\beta}
\newcommand{\g}{\gamma}

\renewcommand{\k}{\kappa}
\newcommand{\s}{\sigma}
\renewcommand{\t}{\tau}

\newcommand{\0}{\emptyset}
\newcommand{\emp}{\emptyset}
\newcommand{\sub}{\subseteq}

\newcommand{\closure}[1]{\overline{#1}}

\newcommand{\card}[1]{\left\lvert #1 \right\rvert}

\newcommand{\continuum}{\mathfrak{c}}
\newcommand{\pseudo}{\mathfrak{p}}
\newcommand{\tower}{\mathfrak{t}}

\newcommand{\ch}{\ensuremath{\mathsf{CH}}\xspace}
\newcommand{\zfc}{\ensuremath{\mathsf{ZFC}}\xspace}

\begin{document}

\title[Factoring a minimal ultrafilter]{Factoring a minimal ultrafilter into a thick part and a syndetic part}
\author{Will Brian}
\address {
W. R. Brian\\
Department of Mathematics and Statistics\\
University of North Carolina at Charlotte\\
9201 University City Blvd.\\
Charlotte, NC 28223, USA}
\email{wbrian.math@gmail.com}
\urladdr{wrbrian.wordpress.com}

\author{Neil Hindman}
\address {
N. Hindman\\
Department of Mathematics\\
Howard University\\
Washington, DC 20059, USA}
\email{nhindman@aol.com}
\urladdr{nhindman.us}


\thanks{Both authors wish to thank Dona Strauss for her insightful 
comments on an earlier draft of this paper, and for providing us with the proofs of
Lemma~\ref{lem:Gdelta} and Theorem~\ref{thm:non-normal}, allowing us to remove some extraneous 
hypotheses from these results.}

\begin{abstract}
Let $S$ be an infinite discrete semigroup. The operation on
$S$ extends uniquely to the Stone-\v Cech compactification $\beta S$
making $\beta S$ a compact right topological semigroup with
$S$ contained in its topological center. As such, $\beta S$ has a smallest
two sided ideal, $K(\beta S)$. An ultrafilter $p$ on $S$ is {\it minimal\/} if
and only if $p\in K(\beta S)$. 

We show that any minimal ultrafilter $p$ factors
into a thick part and a syndetic part. That is, there exist filters $\F$ and 
$\G$ such that $\F$ consists only of thick sets, $\G$ consists only of 
syndetic sets, and $p$ is the unique ultrafilter containing $\F\cup\G$.

Letting $L=\widehat\F$ and $C=\widehat\G$, the sets of ultrafilters containing
$\F$ and $\G$ respectively, we have that $L$ is a minimal left ideal of $\beta S$,
$C$ meets every minimal left ideal of $\beta S$ in exactly one point, and
$L\cap C=\{p\}$. 
We show further that $K(\beta S)$ can be partitioned into
relatively closed sets, each of which meets each minimal left ideal in exactly
one point. 

With some weak cancellation assumptions on $S$, one has
also that for each minimal ultrafilter $p$,
 $S^*\setminus\{p\}$ is not normal. In
particular, if $p$ is a member of either of the disjoint sets $K(\beta\ben,+)$
or $K(\beta\ben,\cdot)$, then $\ben^*\setminus\{p\}$ is not normal.
\end{abstract}

\maketitle


\section{Introduction}

Throughout this paper $S$ will denote an infinite discrete semigroup
with operation $\cdot$.  The Stone-\v Cech compactification $\beta S$
of $S$ is the set of ultrafilters on $S$, with the principal ultrafilters
being identified with the points of $S$. We let $S^*=\beta S\setminus S$. The operation $\cdot$ extends to 
$\beta S$ so that $(\beta S,\cdot)$ is a right topological semigroup,
meaning that for each $p\in\beta S$, the function $\rho_p$ defined
by $\rho_p(q)=q\cdot p$ is continuous, with $S$ contained in the topological center,
meaning that for each $x\in S$, the function $\lambda_x$ defined by 
$\lambda_x(q)= x\cdot q$ is continuous.  Given $p,q\in \beta S$ and
$A\subseteq S$, we have $A\in p\cdot q$ if and only if $\{x\in S:x^{-1}A\in q\}\in p$,
where $x^{-1}A=\{y\in S:x\cdot y\in A\}$.

As does any compact Hausdorff right topological semigroup, 
$\beta S$ has a smallest two sided ideal, $K(\beta S)$.
According to the structure theorem \cite[Theorem 1.64]{H&S}), we have
\begin{align*}
K(\b S) & \,=\, \textstyle \bigcup \set{L \sub \b S}{L \text{ is a minimal left ideal}} \\
& \,=\, \textstyle \bigcup \set{R \sub \b S}{R \text{ is a minimal right ideal}},
\end{align*}
where each of these unions is a disjoint union. The minimal left ideals are closed while the
minimal right ideals are usually not closed. 
Furthermore, if $L$ is a minimal left ideal and $R$ is a minimal right ideal then
\begin{itemize}
\item $L \cap R  = R \cdot L \neq \0$; 
\item $L \cap R$ is a group, and it contains exactly one element of the set $E(R)$ of
idempotents in $R$, namely the identity of the group.
\end{itemize}
In fact, the structure theorem says more than this, but this summary is sufficient for what follows. Furthermore,
\begin{itemize}
\item If $G = L \cap R$, then the map $(p,e) \mapsto p \cdot e$ is a bijection $G \times E(R) \to R$.
\end{itemize}
This last assertion follows from \cite[Theorem 2.11(b)]{H&S},
which asserts that if $L'$ is a minimal left ideal of $\beta S$ and $e$ is the identity of
$L'\cap R$, then the restriction of $\rho_e$ to $G = L \cap R$ is an isomorphism and
a homeomorphism onto $L'\cap R$. 
The idempotent ultrafilters in $K(\b S)$ are called \emph{minimal idempotents} and
the elements of $K(\beta S)$ are called \emph{minimal ultrafilters}.

We will show in Theorem~\ref{thm:transversals} that if $L$ is a minimal left ideal
of $\beta S$, $R$ is a minimal right ideal, $p\in L\cap R$, and
$C=\overline{p\cdot E(R)}$, then $L\cap C=\{p\}$ and $C$ meets each each minimal left ideal
in exactly one point. Further, $\set{q \cdot E(R')}{R' \text{ is a minimal right ideal and } 
q \in L \cap R'}$ partitions $K(\beta S)$\break
 into relatively closed sets.
The fact that the partition elements are closed in $K(\b S)$ can be seen as a topological addition to the 
(algebraic) structure theorem described above. 
Particularly, in the final bullet point, our result shows that the given bijection has at least one 
nice topological property: the images of the ``vertical sections'' $\{p\} \times E(R)$ of 
$G \times E(R)$, namely the sets of the form $p \cdot E(R)$, are closed in $R$. 
(Note that the images of horizontal sections are also closed in $R$, but 
this is not difficult to prove; it follows from the fact that minimal left ideals of $\b S$ are closed.)

Closed subsets of $\b S$ correspond naturally to filters on $S$. For a filter
$\F$ on $S$, let $\widehat\F=\{p\in\beta S:\F\subseteq p\}=\bigcap\{\overline{A}:A\in\F\}$. Given any nonempty subset
$X$ of $\beta S$, $\bigcap X$ is a filter, and if $\F=\bigcap X$, 
then $\widehat\F=\overline{X}$. In terms of filters, our results show that every 
minimal ultrafilter $p$ on $S$ can be ``factored'' into two filters $\F$ and $\G$, where $\F$ consists entirely of thick sets and $\G$ 
consists entirely of syndetic sets. The ultrafilter $p$ factors into $\F$ and $\G$ in 
the sense that $p$ is the filter generated by $\F \cup \G$.

One immediate consequence of this factorization is that every minimal ultrafilter $p$ on $\N$ is a 
butterfly point of $\N^*$. (When we refer to a minimal ultrafilter on $\N$ without
specifying the operation, we mean a member of $K(\beta\ben,+)$.) Recall that a \emph{butterfly point} of a space $X$ is a 
point $p$ such that, for some $A,B \sub X \setminus \{p\}$, we have $\closure{A} \cap \closure{B} = \{p\}$. 
It is an open problem whether every point of $\N^*$ is a butterfly point (e.g., it is 
``classic problem IX" in Peter Nyikos's \emph{Classic Problems in Topology} series \cite{Nyikos}). 

With a little more work, we show that every minimal ultrafilter $p$ is a \emph{non-normality point} of $\N^*$, which means that $\N^* \setminus \{p\}$ is not normal. It is a longstanding open problem whether every point of $\N^*$ is a non-normality point (e.g., it is 
problem 3 on Jan van Mill's list of open problems in \cite{vanMill}). This problem is closely related to the one mentioned in the previous paragraph, because every non-normality point is also a butterfly point. 
It is known that the answer to both problems is consistently positive: for example, \ch implies that every 
point of $\N^*$ is a non-normality point. (This is due to Rajagopalan \cite{Raj} and Warren \cite{Warren} independently.)
It is also known that, using only \zfc, at least some points of $\N^*$ are non-normality points: for example, 
this holds when $p$ is not Rudin-Frol\'ik minimal \cite{B&S}. Our results add to the list of known non-normality points of $\N^*$. 

The result that minimal ultrafilters are butterfly points (respectively, non-normality points) will be proved in a general setting: it holds in $S^*$ whenever $S$ satisfies certain cancellation properties. 
Under the additional assumption that $S$ is countable, we also prove that for a minimal right ideal $R$ of $\beta S$ and
any minimal ultrafilter $p \in R$, the spaces $E(R)$ and $p\cdot E(R)$ are $P$-spaces and not Borel in $\b S$.

\section{Closed transversals and factoring a minimal ultrafilter}

In this section we establish results that do not require any 
cancellation assumptions about $S$, beginning by producing
closed transversals for the set of minimal left ideals.
(By a {\it transversal\/} for this set, we mean a 
set which meets each minimal left ideal in exactly one point.)

A fact that we will use repeatedly is that if $R$
is a minimal right ideal of $\beta S$ and 
$e\in E(R)$, then $e$ is a left identity for 
$R$, which means that $e \cdot p = e$ for all $p \in R$. (In particular, if $e,f\in E(R)$ then $e\cdot f=f$ and
$f\cdot e=e$.) To see this, note that $e \cdot \beta S$ is
a right ideal contained in $R$, so
$e\cdot\beta S=R$ by minimality. Thus $p\in R$ implies
$p=e\cdot q$ for some $q\in\beta S$, so that
$p = e \cdot q = e\cdot e\cdot q = e\cdot p$.

\begin{theorem}\label{thm:transversals}
Let $L$ and $R$ be minimal left and right ideals of $\b S$, respectively, and let $p \in L \cap R$. Then
$$L \cap \closure{p \cdot E(R)} = \{p\}.$$
Furthermore, 
\begin{enumerate}
\item If $L'$ is any minimal left ideal of $\b S$, then
$$L' \cap \closure{p \cdot E(R)} = \{p \cdot e\}$$
where $e$ is the (unique) idempotent contained in $L' \cap R$. In particular, $\closure{p \cdot E(R)}$ meets every minimal left ideal in exactly one point.
\item $\set{q \cdot E(R)}{q \in L \cap R}$ is a partition of $R$ into relatively closed sets (i.e., they are closed in $R$), and
$$\set{q \cdot E(R')}{R' \text{ is a minimal right ideal and } q \in L \cap R'}$$
is a partition of $K(\b S)$ into relatively closed sets.
\end{enumerate}
\end{theorem}
\begin{proof}
Let $p$ be a minimal ultrafilter in $\b S$, let $L$ and $R$ denote the minimal left 
and right ideals of $\b S$, respectively, that contain $p$, and let $f$ be the identity of
$L\cap R$. Then $p=p\cdot f$ so $p\in L \cap \closure{p \cdot E(R)}$.

Suppose $q \in L \cap \closure{p \cdot E(R)}$. We will show that $p=q$.
Let $R'$ denote the minimal right ideal of $\b S$ containing $q$. 
For each $e \in E(R)$, we have $e \cdot f=f$ so $p \cdot e \cdot f = p \cdot f = p$.
Thus the function $\rho_f$ is constant on the set $p \cdot E(R)$, with value $p$. 
But $\rho_f$ is continuous on all of $\b S$, so this means that $\rho_f$ is constant on 
$\closure{p \cdot E(R)}$ with value $p$. In particular, $q \cdot f = p$.
Because $R'$ is a right ideal containing $q$ we have $q \cdot f \in R'$; but $p \in R$, 
so it follows that $R' = R$. Thus $q$ and $p$ are both members of the group $L \cap R$. 
As $f$ is the identity element of this group, $q \cdot f = p$ implies $q=p$, as desired, completing the proof that $L \cap \closure{p \cdot E(R)} = \{p\}$.

To prove $(1)$, suppose $L'$ is any minimal left ideal of $\b S$, and let $e$ denote the identity element of the group $L' \cap R$. 
Clearly $p \cdot e \in p  \cdot  E(R)$, so that $p \cdot e \in L' \cap \closure{p \cdot E(R)}$, and we wish to 
show that it is the only element of this set. Suppose $q \in L' \cap \closure{p \cdot E(R)}$. Exactly as in the 
proof above, we may show that the function $\rho_e$ is constant on $\closure{p \cdot E(R)}$ with value $p \cdot e$. 
Thus $q \cdot e = p \cdot e = p \cdot e \cdot e$. Observe that $p \cdot e \in L' \cap R$ 
(it is in $L'$ because $e \in L'$ and $L'$ is a left ideal, and it is in $R$ because $p \in R$ and $R$ is a 
right ideal). But $L' \cap R$ is a group with identity element $e$, so, if we know that
$q\in L'\cap R$, then $q \cdot e = p \cdot e \cdot e$ implies $q = p \cdot e$, as desired.
We have that $q\in L'$. To see that $q\in R$, let $R'$ be the minimal right ideal with
$q\in R'$. Since $q\cdot e=p\cdot e$, we have that $R=R'$.

To prove $(2)$, let $G = L \cap R$ and let $h: G \times E(R) \to R$ be the function $h(q,e) = q \cdot e$. 
We noted in the introduction that $h$ is a bijection, which implies that $\set{q \cdot E(R)}{q \in G}$ is 
a partition of $R$, which implies that 
$\set{q \cdot E(R')}{R' \text{{ is a minimal right ideal and }} q \in L \cap R'}$ is a partition of 
$K(\b S)$. Finally, all sets of the form $q \cdot E(R)$ are closed in $K(\b S)$, because any point of 
$\big(K(\b S) \cap \closure{q \cdot E(R)}\big) \setminus q \cdot E(R)$ would be a 
member of some minimal left ideal, and this contradicts $(1)$.
\end{proof}

Given a set $X$, we let $\pf(X)$ be the set of finite nonempty
subsets of $X$. A subset $A$ of $S$ is called 
\begin{itemize}
\item \emph{thick} if for each $F\in\pf(S)$, there exists
$x\in S$ such that $Fx\subseteq A$, or, equivalently, if
 the collection of all sets of the form
$\set{s^{-1}A}{s \in S}$
has the finite intersection property.
\item \emph{syndetic} if there is some  $F\in\pf(S)$ such that 
$\textstyle S = \bigcup_{s \in F}s^{-1}A$.
\end{itemize}
Notice that if $A$ is thick and $B$ is syndetic, then $A\cap B\neq
\emp$. (To see this, pick $F\in\pf(S)$ such that $S = \bigcup_{s \in F}s^{-1}B$
and pick $x\in S$ such that $Fx\subseteq A$. Pick $s\in F$ such that
$sx\in B$. Then $sx\in A\cap B$.)

For the semigroup $(\N, + )$, $A \sub \N$ is thick if and only if it 
contains arbitrarily long intervals, and is syndetic if and only if it 
has bounded gaps, which means that there is some $k \in \N$ such that every 
interval of length $k$ contains a point of $A$.

Let $\Theta$ denote the family of thick subsets of $S$, and let $\Sigma$ denote the 
family of syndetic subsets of $S$. These two families of sets are dual to each other, in the following sense, 
which follows immediately from the definitions.

\begin{lemma}\label{lem:duality}
A set is thick if and only if its complement fails to be syndetic, and it is syndetic if and only if its complement fails to be thick.
\end{lemma}

The families $\Theta$ and $\Sigma$ are related to $K(\b S)$ by the following lemma.

\begin{lemma}\label{lem:BHM}
 If $A \sub S$ then
\begin{enumerate}
\item $A \in \Theta$ if and only if $\closure{A}$ contains a minimal left ideal of $\b S$.
\item $A \in \Sigma$ if and only if $\closure{A}$ meets every minimal left ideal of $\b S$.
\end{enumerate}
\end{lemma}
\begin{proof}
This is part of \cite[Theorem 2.9]{BHM}, or see \cite[Theorem 4.48]{H&S}.
\end{proof}

Let us say that a filter $\F$ on $S$ is \emph{$\Theta$-maximal} if $\F \sub \Theta$, and if every filter 
properly extending $\F$ contains some set not in $\Theta$. Similarly, 
let us say that a filter $\G$ on $S$ is \emph{$\Sigma$-maximal} if $\G \sub \Sigma$, and if 
every filter properly extending $\G$ contains some set not in $\Sigma$.
The existence of $\Theta$-maximal filters and $\Sigma$-maximal filters is ensured by Zorn's Lemma.

Note that $\Theta$-maximal filters on $\N$ are never ultrafilters: for example, they will 
contain neither the set of even numbers nor the set of odd numbers. Neither are $\Sigma$-maximal 
ultrafilters on $\N$ ever maximal. In fact, if one identifies subsets of $\N$ with points of the Cantor space 
via characteristic functions, then one can show that $\Sigma$ is a meager, measure-zero subset of the 
Cantor space. Hence every $\Sigma$-maximal filter on $\N$ is also meager and null; 
in this sense, these filters are very far from being ultrafilters.

\begin{lemma}\label{lem:thetamaximal}
A filter $\F$ on $S$ is a $\Theta$-maximal filter if and only if $\widehat \F$ is a minimal left ideal of $\b S$.
\end{lemma}
\begin{proof}
This is \cite[Proposition 3.2]{Brian2}.
\end{proof}

\begin{lemma}\label{lem:BHMagain}
Let $\F$ be a filter on $S$. Then
\begin{enumerate}
\item $\F \sub \Theta$ if and only if $\widehat \F$ contains a minimal left ideal.
\item $\F \sub \Sigma$ if and only if $\widehat \F$ meets every minimal left ideal.
\end{enumerate}
\end{lemma}
\begin{proof}
If $\F \sub \Theta$, then, by an application of Zorn's Lemma, $\F$ can be extended to a $\Theta$-maximal filter $\G$. But then $\widehat \F \supseteq \widehat \G$, so $\widehat \F$ contains a minimal left ideal by Lemma~\ref{lem:thetamaximal}. This proves the ``only if'' direction of $(1)$.

If $\F \not\sub \Theta$, then there is some $A \in \F \setminus \Theta$. But then $\closure{A} \supseteq \widehat \F$, so $\widehat \F$ contains no minimal left ideals by Lemma~\ref{lem:BHM}. This proves the ``if'' direction of $(1)$.

The ``if'' direction of $(2)$ is proved just as it was for $(1)$. Supposing $\F \not\sub \Sigma$, there is some $A \in \F \setminus \Sigma$. But then $\closure{A} \supseteq \widehat \F$, so $\widehat \F$ fails to meet some minimal left ideal by Lemma~\ref{lem:BHM}. 

For the ``only if'' direction of $(2)$, suppose $\F \sub \Sigma$ and let $L$ be any minimal left ideal. $L$ is closed in $\b S$, hence compact, and 
$\set{\closure{A} \cap L}{A \in \F}$
is a collection of closed subsets of $L$ with the finite intersection property (by Lemma~\ref{lem:BHM}, because $\F \sub \Sigma$). Thus
$$\textstyle \widehat \F \cap L = \left( \bigcap \set{\closure{A}}{A \in \F} \right) \cap L = \bigcap \set{\closure{A} \cap L}{A \in \F} \neq \0$$
by compactness. As $L$ was arbitrary, $\widehat \F$ meets every minimal left ideal.
\end{proof}

In light of this lemma, one might hope that the $\Sigma$-maximal filters correspond 
precisely to closed transversals for the set of minimal left ideals, in the same way that $\Theta$-maximal filters correspond 
to the minimal left ideals themselves. We show in Section~\ref{sec:boo} below that this is at least consistently not the 
case. However, the transversals that we found in Theorem~\ref{thm:transversals} do all correspond to $\Sigma$-maximal filters:

\begin{lemma}\label{lem:sigmamaximal}
Let $R$ be a minimal right ideal of $\b S$, let $p \in R$, and let $\G=\bigcap\big( p\cdot E(R)\big)$. Then $\G$ 
is a $\Sigma$-maximal filter on $S$.
\end{lemma}
\begin{proof} By Theorem \ref{thm:transversals} $\widehat\G=\overline{p\cdot E(R)}$ meets every 
minimal left ideal, so by Lemma \ref{lem:BHMagain}, $\G\subseteq \Sigma$. Now
suppose we have a filter $\H\subseteq\Sigma$ such that $\G\subsetneq\H$ and pick
$A\in\H\setminus\G$.  Since $A\notin \G$, pick $f\in E(R)$ such that
$A\notin p\cdot f$. Let $L=\beta S\cdot f$. By Lemma \ref{lem:BHMagain}, $\widehat\H\cap L\neq\emp$
so pick $q\in\widehat\H\cap L$. Since $\widehat \H \subseteq \widehat\G$, $q\in\overline{p\cdot E(R)}$.

Since $A\notin p\cdot f$, $q\neq p\cdot f$, contradicting
Theorem \ref{thm:transversals}(1).\end{proof}

\begin{theorem}\label{thm:factor}
Let $p$ be a minimal ultrafilter on $S$. Then there exist
a $\Theta$-maximal filter $\F$ and
a $\Sigma$-maximal filter $\G$
such that $p$ is the ultrafilter generated by $\F \cup \G$.
Specifically, if $L$ and $R$ are respectively the minimal
left and right ideals of $\beta S$ containing $p$,
then $\F=\bigcap L$ and $\G=\bigcap\big(p\cdot E(R)\big)$ are two such filters.

Moreover, $\F$ is the only $\Theta$-maximal filter contained in $p$, and
$$\F = \set{A \in p}{s^{-1}A \in p \ \mathrm{ for \text{ } all } \ s \in S}.$$
\end{theorem}
\begin{proof}
Let $p$ be a minimal ultrafilter on $S$. Let $L$ and $R$ denote respectively the minimal left and right 
ideals of $\b S$ containing $p$. Let $\F =\bigcap L$ and let $\G = \bigcap\big(p \cdot E(R)\big)$. 

$\F$ is $\Theta$-maximal by Lemma~\ref{lem:thetamaximal} and $\G$ is $\Sigma$-maximal by Lemma~\ref{lem:sigmamaximal}. 
Since any thick set meets any syndetic set, $\F\cup\G$ generates a filter
$\U$. Then $\emp\neq\widehat \U \subseteq \widehat \F \cap \widehat G = L \cap \overline{p\cdot E(R)}
=\{p\}$ by Theorem \ref{thm:transversals}. Hence $\widehat \U = \{p\}$, and this means $\U = p$.

To prove the ``moreover'' assertion of the theorem, suppose $\F'$ is any $\Theta$-maximal filter contained in $p$. Then $p \in \widehat \F'$, and $\widehat \F'$ is a minimal left ideal by Lemma~\ref{lem:thetamaximal}. This implies $\widehat \F' = L$, because the minimal left ideals of $\b S$ are disjoint. $\widehat \F' = L = \widehat \F$ implies $\F' = \F$, so $\F$ is the only $\Theta$-maximal filter contained in $p$.

It remains to show $\F = \set{A \in p}{s^{-1}A \in p \text{ for all }s \in S}$. 
Let $\H=\set{A \subseteq S}{s^{-1}A \in p \text{ for all }s \in S}$.
By \cite[Theorem 6.18]{H&S}, $\widehat\H=\beta S\cdot p=L$. Since
$p\in L$, $\H=\set{A \in p}{s^{-1}A \in p \text{ for all }s \in S}$.
Since $\widehat \H=L=\widehat\F$, $\H=\F$.\end{proof}

While a minimal ultrafilter contains exactly one $\Theta$-maximal filter by the previous theorem, 
we see now that it may contain more than one $\Sigma$-maximal filter. 

\begin{theorem}\label{thm:moreSigma} There is a minimal ultrafilter on $\ben$ that
contains more than one $\Sigma$-maximal filter.
\end{theorem}

\begin{proof} Let $E$ be the set of even numbers, let $O$ be the set of odd numbers, 
let $A=\bigcup_{n=0}^\infty\{2^{2n},2^{2n}+1,2^{2n}+2,\ldots,2^{2n+1}-1\}$,
and let $B = (E \cap A) \cup (O \setminus A)$. Then $B$ has no gaps longer than
$2$, so $B$ is syndetic. By a routine 
application of Zorn's Lemma, there is a $\Sigma$-maximal filter $\H$ such that $B \in \H$. Let $L$ be a 
minimal left ideal. Then $\widehat \H \cap L \neq \0$ by Lemma~\ref{lem:BHMagain}; thus there is a 
minimal ultrafilter $p \in \widehat \H \cap L$. 

Let $R$ be the minimal right ideal with $p\in R$ and let $\G=\bigcap\big(p+ E(R)\big)$.
By Theorem \ref{thm:factor}, $\G$ is $\Sigma$-maximal and $\G\subseteq p$.
We claim that $\G\neq\H$. To see this, note that $E(R)\subseteq E^*$ so if $p\in E^*$, then
$p+E(R)\subseteq E^*$ so $E\in\G$. If $p\in O^*$, then
$p+E(R)\subseteq O^*$ so $O\in\G$. But $B\in \H$ and neither $B\cap E$ nor
$B\cap O$ is syndetic so neither $E$ nor $O$ is a member of $\H$. \end{proof}

To end this section, we will demonstrate a technique for building $\Sigma$-maximal filters on $\N$ that offers some control over the filter obtained (more control, anyway, than is given by Zorn's Lemma). This accomplishes three things. One is to demonstrate that there are many closed transversals for the set of minimal left ideals other than the ones of the form $\closure{p \cdot E(R)}$. Another is to lay the foundation for Section 4, which uses a few of the following lemmas. The third is an improvement on Theorem~\ref{thm:moreSigma}: we will show that every minimal ultrafilter on $\ben$ contains more than one $\Sigma$-maximal filter.


\begin{lemma}\label{lem:manyleft} Let $n\in\ben$ and assume $\langle X_i\rangle_{i=1}^n$ is a 
sequence of pairwise disjoint subsets of $S$ such that
$\big(\forall G\in\pf(S)\big)\big(\exists H\in\pf(S)\big)(\forall x\in S)\break
(\exists y\in S)(\exists i\in\nhat{n})\big(Gy\subseteq (Hx\cap X_i)\big)$.
Then for each minimal left ideal $L$ of $\beta S$, there is some $i\in\nhat{n}$
such that $L\subseteq \overline{X_i}$.\end{lemma}

\begin{proof} Let $L$ be a minimal left ideal of $\beta S$. Aiming for a contradiction, suppose
that for each $i\in\nhat{n}$, $L\setminus\overline{X_i}\neq\emp$. Let
$\F=\bigcap L$. By Lemma \ref{lem:thetamaximal}, $\F$ is $\Theta$-maximal.
We claim that for each $i$, there exists $B_i\in\F$ such that
$B_i\cap X_i$ is not thick. If $L\cap\overline{X_i}=\emp$, one may
ket $B_i=S\setminus X_i$. If $L\cap\overline{X_i}\neq\emp$, then
$\G_i=\bigcap (L\cap\overline{X_i})=\{C\subseteq S:(\exists B\in\F)(B\cap X_i\subseteq C)\}$
is a filter properly containing $\F$. (The containment is proper because $L\setminus\overline{X_i}\neq\emp$.) Hence one may pick $B_i\in\F$ such that
$B_i\cap X_i$ is not thick.

For each $j\in\nhat{n}$, let $D_j=(\bigcap_{i=1}^n B_i)\cap X_j$.
Then for $j\in\nhat{n}$, $D_j$ is not thick so pick $G_j\in\pf(S)$ such that
for all $y\in S$, $G_jy\not\subseteq D_j$. Let $G=\bigcup_{j=1}^n G_j$ and 
pick $H\in\pf(S)$ as guaranteed by the hypothesis. 
Now $\bigcap_{i=1}^n B_i\in\F$, so in particular $\bigcap_{i=1}^n B_i$ is thick. Pick
$x\in S$ such that $Hx\subseteq \bigcap_{i=1}^n B_i$. Pick
$y\in S$ and $j\in\nhat{n}$ such that $Gy\subseteq (Hx\cap X_j)$.
Then $G_jy\subseteq D_j$, a contradiction.
\end{proof}

Note that, as a consequence of the following theorem, for each $n\in\ben$,
there is a partition of $K(\beta\ben)$ into $n$ sets, each clopen in $K(\beta\ben)$,
so that every minimal left ideal is contained in one cell of the partition.

\begin{theorem}\label{thm:npartition} Let $n\in\ben$ and let
$\langle Z_j\rangle_{j=1}^n$ be a partition of
$\ben$. Let $\langle I_t\rangle_{t=1}^\infty$ be a partition of
$\ben$ into intervals such that $\displaystyle \lim_{t\to\infty}|I_t|=\infty$.
For $j\in \nhat{n}$, let $X_j=\bigcup_{t\in Z_j}I_t$.
Then for each for each minimal left ideal $L$ of 
$\beta\ben$, there exists $j\in\nhat{n}$ such that $L\subseteq\overline{X_j}$.
If $Z_j$ is infinite, then $\overline{X_j}$ contains a minimal
left ideal of $\beta\ben$. \end{theorem}

\begin{proof}  We may presume that for each $t\in\ben$, $\max I_t+1=\min I_{t+1}$.
If $Z_j$ is infinite, then $X_j$ is thick, so the second conclusion is immediate.
To establish the first conclusion we invoke Lemma \ref{lem:manyleft}. That  is, we show that
$$\begin{array}{l}\big(\forall G\in\pf(\ben)\big)\big(\exists H\in\pf(\ben)\big)(\forall x \in\ben)
(\exists y\in\ben)\\(\exists i\in\nhat{n})(G+y\subseteq(H+x)\cap X_i)\,.\end{array}$$ So
let $G\in\pf(\ben)$ and let $k=\max G$. Pick $M\in\ben$ such that for all
$n\geq M$, the length of $I_n$ is at least $k$ and let $m=\max I_M$.
Let $H=\nhat{m+k}$ and let $x\in\ben$. Pick the largest $n$ such that $z=\max I_n\leq x+m$ and note
that $n\geq M$ so that the length of $I_n$ and the length of $I_{n+1}$ are both
at least $k$ and thus $\{z-k+1,z-k+2,\ldots,z\}\subseteq I_n$ and
$\{z+1,z+2,\ldots,z+k\}\subseteq I_{n+1}$. If $z-k\geq x$, let $y=z-k$ so that
$G+y\subseteq\{y+1,y+2,\ldots,y+k\}\subseteq\{x+1,x+2,\ldots,x+m+k\}\cap I_n=H+x\cap I_n$.
If $z-k<x$, let $y=z$ so that
$G+y\subseteq\{y+1,y+2,\ldots,y+k\}\subseteq\{x+1,x+2,\ldots,x+m+k\}\cap I_{n+1}=H+x\cap I_{n+1}$.
\end{proof}

We remark that if $S$ is the free semigroup on a finite alphabet (where the operation $\cdot$ is concatonation), 
if $n\in\ben$, $X_j$ is as in Theorem \ref{thm:npartition} for $j\in \nhat{n}$,
and $Y_j=\{w\in S:\text{ the length of }w\text{ is in }X_j\}$,
then each minimal left ideal of $\beta S$ is contained in 
$\overline{Y_j}$ for some $j\in\nhat{n}$. We leave the details to the
reader.

\begin{theorem}\label{thm:manyright} Let $\RR$ be a finite set of minimal right ideals of $\beta\ben$.
There is a $\Sigma$-maximal filter $\G$ on $\ben$ such that $\widehat \G$ is a closed transversal
for the minimal left ideals of $\beta\ben$, $K(\b\N) \cap \widehat \G\subseteq\bigcup\RR$, and $\widehat\G\cap R\neq\emp$ for every $R\in\RR$.
Furthermore, if $p$ is any minimal ultrafilter contained in one of the members of $\RR$, than we may find such a filter $\G$ with $p \in \widehat\G$.
\end{theorem}

\begin{proof} Enumerate $\RR$ as $\langle R_i\rangle_{i=1}^n$, and fix $p \in R_1$.
Let 
$\langle X_j\rangle_{j=1}^n$ be as in Theorem \ref{thm:npartition}, assuming that
each $Z_j$ is infinite. Without loss of generality (by relabelling the $Z_j$ if necessary) we may assume that $p \in \closure{X_1}$.
Let 
$$\G = \textstyle \bigcap\big(\big((p+E(R_1)) \cap \closure{X_1}\big) \cup \bigcup_{i=2}^n(E(R_i)\cap\overline{X_i})\big).$$
We show first
that if $L$ is a minimal left ideal of $\beta\ben$, $i\in\nhat{n}$, and
$L\subseteq\overline{X_i}$, then either
\begin{itemize}
\item $i=1$ and $\widehat\G \cap L = \{p+f\}$, where $f$ is the identity of $L \cap R_1$, or
\item $i>1$ and $\widehat\G \cap L = \{f\}$, where $f$ is the identity of $L \cap R_i$.
\end{itemize}
This will establish that $\widehat \G$ 
is a transversal for the minimal left ideals of $\beta\ben$ and that $p \in \widehat\G$. It will
also establish that $\widehat\G \cap R_i\neq\emp$ for each $i\in\nhat{n}$,
because each $X_i$ is thick, which implies that for each $i$ there is some minimal left ideal $L$ with $L \sub \closure{X_i}$.

Observe that 
\begin{equation}\tag{$*$}
\begin{array}{rl}
\widehat\G&=
\textstyle \closure{\big((p+E(R_1)) \cap \closure{X_1}\big) \cup \bigcup_{i=2}^n(E(R_i)\cap\overline{X_i})} \\
&= \textstyle \closure{(p+E(R_1)) \cap \closure{X_1}} \cup \bigcup_{i=2}^n \closure{E(R_i)\cap\overline{X_i}} \\
&= \textstyle \big( \closure{p+E(R_1)} \cap {\closure{X_1}} \big) \cup \bigcup_{i=2}^n (\overline{E(R_i)}\cap\overline{X_i}) \,.
\end{array}
\end{equation}
(The third line follows from the second because the $\closure{X_i}$ are not only closed, but clopen.)

For the first bullet point, suppose $i = 1$, let $L \sub \closure{X_1}$ be a minimal left ideal, and let $f$ be the identity of $L \cap R_1$.
By 
Theorem \ref{thm:transversals}, $L\cap\overline{p+E(R_1)}
=\{p+f\}$. 
Since $L\subseteq\overline{X_1}$, and since $\closure{X_j} \cap \closure{X_1} = \0$ for $j \neq 1$, $(*)$ implies that $\widehat\G\cap L=\overline{p+E(R_1)}\cap L=\{p+f\}$.

For the second bullet point, suppose $i \neq 1$, let $L \sub \closure{X_i}$ be a minimal left ideal, and let $f$ be the identity of $L \cap R_i$.
By 
Theorem \ref{thm:transversals}, $L\cap\overline{E(R_i)}
=\{f\}$. 
Since $L\subseteq\overline{X_i}$, and since $\closure{X_j} \cap \closure{X_i} = \0$ for $j \neq i$, $(*)$ implies that $\widehat\G\cap L=\overline{E(R_i)}\cap L=\{f\}$.

$\widehat\G$ meets every minimal left ideal, so $\G\subseteq\Sigma$ by Lemma~\ref{lem:BHMagain}. To finish the proof, we must show that $\G$ is $\Sigma$-maximal.
Aiming for a contradiction, suppose that $\H$ is a filter contained in $\Sigma$ which properly
contains $\G$ and pick $A\in\H\setminus\G$. Since $A\notin\G$, pick
$$f \in \textstyle \big((p+E(R_1)) \cap \closure{X_1}\big) \cup \bigcup_{i=2}^n(E(R_i)\cap\overline{X_i})$$
such that $A \notin f$. Either
\begin{itemize}
\item $f \in (p+E(R_1)) \cap \closure{X_1}$, or
\item $f \in E(R_j)\cap\overline{X_j}$ for some $j \neq 1$.
\end{itemize}
In either case, let $L=\beta\ben+f$. 
By Lemma \ref{lem:BHMagain}, $L\cap\widehat\H\neq\emp$ so
pick $q \in L\cap\widehat\H$. 
Since $A \in q$ we have $q \neq f$.
But $q \in L\cap\widehat\H\subseteq L\cap\widehat\G = \{f\}$, a contradiction.
\end{proof}

\begin{corollary}
Every minimal ultrafilter on $\ben$ contains more than one $\Sigma$-maximal filter.
\end{corollary}
\begin{proof}
Let $p$ be a minimal ultrafilter, let $R$ be the minimal right ideal containing $p$, and let $R'$ be any other minimal right ideal. By Theorem~\ref{thm:factor}, $\G = \bigcap (p+E(R))$ is a $\Sigma$-maximal filter contained in $p$. By the previous theorem, there is a $\Sigma$-maximal filter $\H$ contained in $p$ such that $\widehat\H \cap R' \neq \0$. Theorem~\ref{thm:transversals} implies that $\widehat\G \cap K(\b\N) \sub p+E(R) \sub R$, so that $\widehat\G \cap R' = \0$. Thus $\G \neq \H$.
\end{proof}

\section{Topology in $K(\b S)$}

A set $F\subseteq S$ is a {\it left solution set\/} (respectively a {\it right solution set\/}) if and only if
there exist $a,b\in S$ such that $F=\{x\in S:ax=b\}$ (respectively $F=\{x\in S:xa=b\}$). 
If every left solution set and every right solution set is finite, then $S$ is called {\it weakly cancellative\/}.  
If $|S|=\kappa$ and the union of fewer than $\kappa$ solution sets (left or right) always
has cardinality less than $\kappa$, then $S$ is called {\it very weakly cancellative\/}. Of 
course, if $\kappa=\omega$, then ``weakly cancellative'' and ``very weakly cancellative'' mean the same thing.
We let $U(S)$ denote the set of uniform ultrafilters on $S$. By \cite[Lemma 6.34.3]{H&S}, if $S$
is very weakly cancellative, then $U(S)$ is an ideal of $\beta S$. 

The easy results of the following lemma do not appear to have been written down before.

\begin{lemma}\label{lem:same} Assume that $S$ is very weakly cancellative. Then
$K(\beta S)=K\big(U(S)\big)$, the minimal left ideals of $\beta S$ and $U(S)$ are the
same, and the minimal right ideals of $\beta S$ and $U(S)$ are the same. If
$L$ is a minimal left ideal of $\beta S$ and $p\in L$, then
$L=\beta S\cdot p=S^*\cdot p=U(S)\cdot p$.\end{lemma}

\begin{proof} Since $U(S)$ is an ideal of $\beta S$, $K(\beta S)\subseteq U(S)$ and
thus by \cite[Theorem 1.65]{H&S}, $K\big(U(S)\big)=K(\beta S)$.

Let $T$ be a minimal left ideal of $U(S)$. Since $U(S)$ is a left ideal of $\beta S$,
by \cite[Lemma 1.43(c)]{H&S}, $T$ is a minimal left ideal of $\beta S$.  Now let
$L$ be a minimal left ideal of $\beta S$. Since $U(S)$ is a right ideal of $\beta S$,
$L\cap U(S)\neq\emp$. Since $U(S)$ is a left ideal of $\beta S$, $L\cap U(S)$ is
a left ideal of $\beta S$ contained in $L$, so $L\cap U(S)=L$. That is, $L\subseteq
U(S)$. Thus $L$ is a left ideal of $U(S)$ so pick a minimal left ideal $T$
of $U(S)$ such that $T\subseteq L$. As we just saw, $T$ is a left ideal of $\beta S$
so $T=L$.

The arguments in the paragraph above were completely algebraic, so by a left-right
switch, we have that the minimal right ideals of $\beta S$ and $U(S)$ are the
same.

Finally, let $L$ be a minimal left ideal of $\beta S$ and let
$p\in L$. Then $\beta S\cdot p$ is a left ideal of $\beta S$ contained
in $L$, so $L=\beta S\cdot p$. Also, $U(S)\cdot p$ is a left ideal of
$U(S)$ contained in $L$, so $L=U(S)\cdot p$. Thus
$L=U(S)\cdot p\subseteq S^*\cdot p\subseteq \beta S\cdot p=L$.
\end{proof}

Note the similarity of this lemma with \cite[Theorems 4.36 and 4.37]{H&S}, which state that $S$ is weakly cancellative if and only if $S^*$ is an ideal of $\b S$, in which case $K(S^*) = K(\b S)$, and a set is a minimal left ideal (respectively, minimal right ideal) for $S^*$ if and only if it is a minimal left ideal (respectively, minimal right ideal) for $\b S$. 

Note that very weak cancellativity does not imply $S^*$ is an ideal of $\b S$ (because this is equivalent to weak cancellativity); in general, it may not even be a sub-semigroup of $\b S$.
However, as an immediate corollary to Lemma~\ref{lem:same}, if $S$ is very weakly cancellative then $K(\b S) \sub S^*$.

All the results of this section (except Lemmas \ref{lem:qER} and \ref{lem:muove}) assume that $S$ is 
very weakly cancellative.  
Under the additional assumption that there is a uniform, finite bound on $|\{x\in S:xa=a\}|$ for $a\in S$,
we establish
that if $p$ is a minimal ultrafilter on $S$, then 
\begin{itemize}
\item $p$ is a butterfly point of $S^*$ and, furthermore,
\item $S^*\setminus\{p\}$ is not normal.
\end{itemize}
Under the additional assumption that $S$ is countable, we also show that if $R$ is the minimal right ideal containing $p$, then
\begin{itemize}
\item $p\cdot E(R)$ is a $P$-space, and
\item $p\cdot E(R)$ is not Borel in $\b S$. 
\end{itemize}
Note in particular that these results apply to the
semigroups $(\ben,+)$ and $(\ben,\cdot)$, which are easily seen to satisfy all of the above assumptions. 
The proofs proceed by extracting the topological content of Theorem~\ref{thm:factor}, 
which provides a canonical (and useful) basis for the space $p \cdot E(R)$. 

\begin{lemma}\label{lem:qER} Let $R$ be a minimal right ideal of $\beta S$ and
let $p\in R$. If $q\in p\cdot E(R)$, then $q\cdot E(R)=p\cdot E(R)$.\end{lemma}

\begin{proof} Assume that $q\in p\cdot E(R)$ and pick $e\in E(R)$ such that
$q=p\cdot e$. Given any $f\in E(R)$, $q\cdot f=p\cdot e\cdot f=p\cdot f$ so 
$q\cdot E(R)\subseteq p\cdot E(R)$. Now let $L$ be the minimal left ideal
with $p\in L$ and let $f$ be the identity of $L\cap R$. Then $q\cdot f=p\cdot f=p$
so $p\in q\cdot E(R)$ so, as above, $p\cdot E(R)\subseteq q\cdot E(R)$.\end{proof}

For $a\in S$, let $\Fix(a)=\{x\in S:xa=a\}$.
Several results in this section use the hypothesis that there is a uniform, finite bound on the size of the sets $\Fix(a)$. The left-right switch of \cite[Theorem 4.11]{HJS} shows that this 
assumption is strictly weaker than the 
assertion that there is a finite bound on the size of right solution sets.

\begin{lemma}\label{lem:muove} Let $k\in\ben$ and assume that for
each $a\in S$, $|\Fix(a)|\leq k$. Then for each 
$p\in S^*$, $|\{x\in S:x\cdot p=p\}|\leq k$.\end{lemma}

\begin{proof} Let $p\in S^*$ and suppose we have distinct $x_1,x_2,\ldots,x_{k+1}$ in $S$
such that $x_ip=p$ for each $i$. For $i\in \nhat{k+1}$, let
$D_i=\{a\in S:x_ia=a\}$. Since each $\lambda_{x_i}$ is continuous,
we have by \cite[Theorem 3.35]{H&S} that each $D_i\in p$.
Pick $a\in\bigcap_{i=1}^{k+1}D_i$. Then
$\{x_1,x_2,\ldots,x_{k+1}\}\subseteq \Fix(a)$, a contradiction. \end{proof}

Let us note that the conclusion of this lemma is not a consequence of very weak cancellativity, or even of weak cacellativity. The 
semigroup $(\N,\vee)$, where $a \vee b = \max\{a,b\}$, is weakly cancellative, but $n \vee p = p$ 
for every $p \in \N^*$ and $n \in \N$.

\begin{theorem}\label{thm:basis}
Suppose $S$ is very weakly cancellative. Let $p$ be a minimal ultrafilter on $S$, and let 
$L$ and $R$ denote the minimal left and right ideals of $\beta S$ containing $p$. Then
$$\set{A^* \cap B^*}{L \sub A^* \text{ and } p \cdot E(R) \sub B^*}$$
is a local basis for $p$ in $S^*$.
\end{theorem}

\begin{proof}
 Let $\F=\bigcap L$ and $\G=\bigcap\big(p\cdot E(R)\big)$.
By Theorem \ref{thm:factor}, $p$ is the filter generated by $\F\cup\G$.
That is, $p=\{C\subseteq S:(\exists A\in\F)(\exists B\in\G)
(A\cap B\subseteq C)\}$. Thus, given $C\in p$, pick  $A\in\F$ and
$B\in\G$ such that $A\cap B\subseteq C$. Then $p\in A^*\cap B^*\subseteq C^*$.
\end{proof}

\begin{corollary}\label{cor:thickbasis}
Suppose $S$ is very weakly cancellative. Then
\begin{enumerate}
\item If $L$ is a minimal left ideal of $\beta S$, then
$\set{B^*\cap L}{B \in \Sigma}$ 
is a basis for $L$.
\item If $R$ is a minimal right ideal and $p \in R$, then
$\set{A^*\cap p\cdot E(R)}{A \in \Theta}$
is a basis for $p \cdot E(R)$.
\end{enumerate}
\end{corollary}

\begin{proof}
To establish (1), assume that $L$ is a minimal left ideal,
let $C\subseteq S$ such that $C^*\cap L\neq\emp$, and pick
$p\in C^*\cap L$. Let $R$ be the minimal right ideal with
$p\in R$. By Theorem \ref{thm:basis} pick $A,B\subseteq S$ such that
$L\subseteq A^*$, $p\cdot E(R)\subseteq B^*$, and
$p\in A^*\cap B^*\subseteq C^*$. By Theorem \ref{thm:transversals},
$\overline{p\cdot E(R)}$ meets every minimal left ideal so by 
Lemma \ref{lem:BHM}, $B\in \Sigma$.
Also $C^*\cap L\supseteq A^*\cap B^*\cap L
=B^*\cap L$.

To establish (2), let $R$ be a minimal right ideal and let
$p \in R$. Let $C\subseteq S$ such that $C^*\cap\big(p\cdot E(R)\big)\neq\emp$,
and pick $q\in C^*\cap\big(p\cdot E(R)\big)$. Let $L$ be the minimal left ideal
with $q\in L$. Then $q\cdot E(R)=p\cdot E(R)$ by Lemma \ref{lem:qER} so by
Theorem \ref{thm:basis} pick $A,B\subseteq S$ such that
$L\subseteq A^*$, $q\cdot E(R)\subseteq B^*$, and
$q\in A^*\cap B^*\subseteq C^*$. Since $L\subseteq A^*$, $A\in \Sigma$ by Lemma \ref{lem:BHM},
and $C^*\cap\big(p\cdot E(R)\big)\supseteq A^*\cap B^*\cap\big(p\cdot E(R)\big)
=A^*\cap\big(p\cdot E(R)\big).$
\end{proof}


\begin{lemma}\label{lem:thickminus} Assume $S$ is very weakly 
cancellative and $|S|=\kappa$. Let $A$ be a thick subset of
$S$ and let $F\subseteq S$ such that $|F|<\kappa$. Then
$A\setminus F$ is thick.\end{lemma}

\begin{proof} Pick a minimal left ideal $L$ of $\beta S$ such that
$L\subseteq \overline{A}$. As we have noted, $U(S)$ is an ideal of $\beta S$, which implies that
$L\subseteq U(S)$. Now $\overline{F}\cap U(S)=\emp$ so
$L\subseteq \closure{A} \setminus \closure{F} = \overline{A\setminus F}$.
\end{proof}

\begin{lemma}\label{lem:leftideals} 
Assume that $S$ is very weakly 
cancellative, $|S|=\kappa$, and $A$ is a thick subset of
$S$. Then $A$ contains $\kappa$ pairwise disjoint thick sets.
\end{lemma}
\begin{proof} 
Enumerate $\kappa\times\kappa$ as $\langle \delta(\sigma),\tau(\sigma)\rangle_{\sigma<\kappa}$
and enumerate $\pf(S)$ as $\langle F_\iota\rangle_{\iota<\kappa}$. We inductively
choose $\langle x_\sigma\rangle_{\sigma<\kappa}$ such that for $\sigma<\kappa$,
$F_{\delta(\sigma)}\cdot x_\sigma\subseteq A$ and for $\mu<\sigma<\kappa$,
$F_{\delta(\sigma)}\cdot x_\sigma\cap F_{\delta(\mu)}\cdot x_\mu=\emp$.  Having
chosen $\langle x_\mu\rangle_{\mu<\sigma}$, let
$H=\bigcup_{\mu<\sigma}F_{\delta(\mu)}\cdot x_\mu$. Then
$|H|<\kappa$ so by Lemma \ref{lem:thickminus}, $A\setminus H$ is thick, so
we may pick $x_\sigma$ with $F_{\delta(\sigma)}\cdot x_\sigma\subseteq A\setminus H$.
Having chosen $\langle x_\sigma\rangle_{\sigma<\kappa}$, for each $\eta<\kappa$,
let $A_\eta=\bigcup\{F_{\delta(\sigma)}\cdot x_\sigma:\tau(\sigma)=\eta\}$.
Then $\langle A_\eta\rangle_{\eta<\kappa}$ is a sequence of $\kappa$ pairwise
disjoint thick subsets of $A$.
\end{proof}


\begin{lemma}\label{lem:butterfly}
Suppose $S$ is very weakly cancellative and there is a uniform, finite bound on
$|\Fix(a)|$ for $a\in S$. Let $L$ and $R$ be minimal left and right ideals of $\beta S$, and let $p \in R$. 
Then neither $L$ nor $p \cdot E(R)$ has any isolated points in the topology they inherit from $S^*$.
\end{lemma}

\begin{proof}
Suppose that $q$ is an isolated point of $L$. Pick $A\in p$
such that $A^*\cap L=\{p\}$. Now $L$ is a minimal left
ideal of $\beta S$ and
$q\in L=S^*\cdot q$ by Lemma~\ref{lem:same}, so pick $r\in S^*$ such that
$q=r\cdot q$. Then $\{x\in S:x^{-1}A\in q\}\in r$.
Let $F=\{x\in S:x\cdot q=q\}$. Then $F$ is finite
by Lemma \ref{lem:muove}, so pick $x\in S\setminus F$ such that
$x^{-1}A\in q$. Then $A\in x\cdot q$ so $x\cdot q\in A^*\cap L$
and $x\cdot q\neq q$, a contradiction.

Suppose $q \in p \cdot E(R)$, and let $U$ be a neighborhood of $q$ in 
$p \cdot E(R)$. By Corollary~\ref{cor:thickbasis}, there is some thick set $A$ such that 
$A^* \cap (p \cdot E(R)) \sub U$. By combining Lemma~\ref{lem:leftideals} with Lemma~\ref{lem:BHM}, $A^*$ contains more than one minimal left ideal. Each minimal left ideal contains a point of $p \cdot E(R)$
by Theorem \ref{thm:transversals}, so this shows that $A^*$, hence $U$, contains more than one point of $p \cdot E(R)$.
\end{proof}

Let us note that weak 
cancellativity alone is not enough to prove this lemma. The semigroup $(\N,\vee)$ is weakly 
cancellative, but $q \vee p = p$ for every $p,q \in \N^*$. This means that $\{p\}$ is a minimal left ideal for every $p \in \N^*$.

\begin{theorem}\label{thm:butterfly}
Suppose $S$ is  very weakly cancellative and has a uniform, finite bound on $|\Fix(a)|$ for $a\in S$. 
Then every minimal ultrafilter on $S$ is a butterfly point of $S^*$.
\end{theorem}
\begin{proof}
Let $L$ and $R$ be the minimal left and right ideals containing $p$.
Theorem~\ref{thm:transversals} asserts that $\{p\} = L \cap \closure{p \cdot E(R)}$. 
Neither $L$ nor $\closure{p \cdot E(R)}$ has any isolated points by Lemma~\ref{lem:butterfly}, so this makes $p$ a butterfly point.
\end{proof}

We have included the proof of Theorem~\ref{thm:butterfly} because of its 
naturalness and simplicity. But we prove next a stronger result that 
supersedes Theorem~\ref{thm:butterfly} by showing, under the same assumptions, that every minimal ultrafilter is a non-normality point of $S^*$.

\begin{lemma}\label{lem:Gdelta}
Let $S$ be a very weakly cancellative semigroup with
$|S|=\kappa$, and let ${\U}$ be
a collection of open subsets of $S^*$ with $|{\U}|\leq\kappa$. If $\bigcap{\U}$ contains a minimal left ideal of $\b S$, then
$\bigcap{\U}$ contains $2^{2^\kappa}$ distinct
minimal left ideals of $\b S$.
\end{lemma}

\begin{proof} 
Let $L$ be a minimal left ideal of $\b S$ with $L \subseteq \bigcap{\mathcal U}$.
We claim that for each $U \in {\mathcal U}$, there exists $B_{U} \subseteq S$
such that $L \subseteq B_U^* \subseteq U$. Let $U\in {\mathcal U}$.
For each $p \in L$ pick $C_p \in p$ such that $C_p^* \subseteq U$.
Using the compactness of $L$, pick a finite $F \subseteq L$ such that $L \subseteq \bigcup_{p\in F}C_p^*$,
and let $B_U = \bigcup_{p \in F} C_p$. Then $B_U^* = \bigcup_{p \in F} C_p^* \subseteq U$, as claimed.
Let 
$${\mathcal B}=\big\{\bigcap{\mathcal F}:{\mathcal F}\in\pf(\{B_U:U\in{\mathcal U}\})\big\}.$$
Observe that ${\mathcal B}$ is a set of at most $\kappa$ subsets of $S$,
$L\subseteq\bigcap_{B\in{\mathcal B}}B^*\subseteq \bigcap{\mathcal U}$, and $\B$ is closed
under finite intersections.

Enumerate $S$ as $\seq{ s_\sigma }{\sigma<\kappa}$ and enumerate
${\mathcal B}\times \pf(S)$ as $\seq{ D_\sigma }{\sigma<\kappa}$. For
$\sigma<\kappa$, let $E_\sigma=\bigcap_{s\in F}s^{-1}B$, where
$(B,F)=D_\sigma$. 

We claim that $|E_\sigma|=\kappa$ for
each $\sigma<\kappa$. To see this, let $p\in L$, let
$\sigma<\kappa$, and let $(B,F)=D_\sigma$. For
each $s\in F$, $s \cdot p\in L\subseteq\overline{B}$, which implies that $s^{-1}B\in p$, which implies that $E_\sigma\in p$. 
From this and [HS, Lemma 6.34.3], it follows that $|E_\sigma|=\kappa$.

We now construct a sequence of elements of $S$ by transfinite recursion.
To begin, pick $t_0\in E_0$.  Given $0<\mu<\kappa$, assume we have chosen
$\seq{t_\sigma}{\sigma<\mu}$ already such that
\begin{enumerate}
\item{} if $\sigma<\mu$, then $t_\sigma\in E_\sigma$,
\item{} if $\sigma<\delta<\mu$, then $t_\sigma\neq t_{\delta}$, and
\item{} if $\sigma<\mu$, $\eta<\sigma$, $\nu<\sigma$, and $\tau<\sigma$, then $s_\eta \cdot t_\nu\neq s_\tau \cdot t_\sigma$.
\end{enumerate}

\noindent Given $\eta<\mu$, $\nu<\mu$, and $\tau<\mu$, let
$A_{\eta,\nu,\tau}=\{t\in S:s_\eta \cdot t_\nu=s_\tau \cdot t\}$.
Then each $A_{\eta,\nu,\tau}$ is a left solution set, so
$|\bigcup_{\eta<\mu}\bigcup_{\nu<\mu}\bigcup_{\tau<\mu}A_{\eta,\nu,\tau}|<\kappa$.
Pick $$\textstyle t_\mu\in E_\mu\setminus(\{t_\sigma:\sigma<\mu\}\cup
\bigcup_{\eta<\mu}\bigcup_{\nu<\mu}\bigcup_{\tau<\mu}A_{\eta,\nu,\tau})\,.$$
The three hypotheses are again satisfied at the next stage of the recursion, and this completes the construction of our sequence $\seq{t_\s}{\s < \k}$.

\begin{claim}
If $p\hbox{ and }q$ are distinct uniform ultrafilters on 
$T=\{t_\sigma:\sigma<\kappa\}$, then $\beta S \cdot p\cap\beta S \cdot q=\emp$. 
\end{claim}
\begin{proof}[Proof of claim]
Assume $P$ and $Q$ are disjoint subsets of $T$, with
$P\in p$ and $Q\in q$.  Then we claim that
$$\textstyle \beta S \cdot p \subseteq \closure{\{s_\eta \cdot t_\sigma:t_\sigma\in P\hbox{ and }\eta<\sigma\}}.$$
To see this, it suffices to show that
$S \cdot p \subseteq \closure{\{s_\eta \cdot t_\sigma:t_\sigma\in P\hbox{ and }\eta<\sigma\}}$. 
Let $s_\nu\in S$. 
As $p$ is uniform, $\{t_\sigma:t_\sigma\in P\hbox{ and }\nu<\sigma\}\in p$,
so that $s_\nu \cdot \{t_\sigma:t_\sigma\in P\hbox{ and }\nu<\sigma\}\in s_\nu \cdot p$
and $s_\nu \cdot \{t_\sigma:t_\sigma\in P\hbox{ and }\nu<\sigma\} \subseteq
\{s_\eta \cdot t_\sigma:t_\sigma\in P\hbox{ and }\eta<\sigma\}$. Similarly,
$$\beta S \cdot q \subseteq \closure{\{s_\nu \cdot t_\delta:t_\delta\in Q\hbox{ and }\nu<\delta\}}.$$
Because $\{s_\eta \cdot t_\sigma:t_\sigma\in P\hbox{ and }\eta<\sigma\} \cap
\{s_\nu \cdot t_\delta:t_\delta\in Q\hbox{ and }\nu<\delta\} = \emp$ by construction, we
have that $\beta S \cdot p \cap \beta S \cdot q = \emp$, as desired.
\end{proof}

Consider the relation on $\set{D_\s}{\s < \k}$ defined by
$$D_\sigma\prec D_\tau \quad \text{ if and only if } \quad \pi_1(D_\tau)\subseteq\pi_1(D_\sigma) \ \text{ and } \ \pi_2(D_\sigma)\subseteq \pi_2(D_\tau)$$
where, as usual, $\pi_1(B,F) = B$ and $\pi_2(B,F) = F$. 
Observe that, by our choice of $\B$ and the definition of the $D_\s$, any finitely many members of $\set{D_\s}{\s < \k}$ have a common upper bound with respect to $\prec$. In other words, $\set{D_\s}{\s < \k}$ is directed by $\prec$.

For each $\sigma<\kappa$, let
$T_{\sigma}=\{t_\tau:D_\sigma\prec D_\tau\}$. We claim that
$\{T_\sigma:\sigma<\kappa\}$ has the $\kappa$-uniform finite intersection
property. (This means that the intersection of finitely many of the $T_\sigma$ has size $\kappa$.) To see this, first observe that each $T_\s$ has size $\k$, because if $T_\s = (B,F)$ then for any $s\in S\setminus F$,
$D_\s \prec (B,F\cup\{s\})$. Then, if $H\in\pf(\kappa)$, pick $\tau$ such that $D_\sigma \prec D_\tau$ for each $\sigma\in H$, and observe that $|\bigcap_{\sigma\in H}T_\sigma| \geq |T_\t| = \kappa$.

By [HS, Theorem 3.62], there are $2^{2^\kappa}$ distinct uniform ultrafilters on $S$
containing $\{T_\sigma:\sigma<\kappa\}$. 
For each such ultrafilter $p$, let $L_p$ denote a minimal left ideal contained in $\b S \cdot p$. (One must exist, because $\b S \cdot p$ is a left ideal.) If $p \neq q$, then $L_p \neq L_q$, because $\b S \cdot p$ and $\b S \cdot q$ are disjoint by the claim above. To complete the proof of the theorem, we will show that each such $L_p$ is contained in $\bigcap \U$.

Let $p$ be a uniform ultrafilter on $S$
containing $\{T_\sigma:\sigma<\kappa\}$.
If $B \in \B$ and $s \in S$, then pick $\sigma<\kappa$ such that $D_\sigma=(B,\{s\})$, and observe that $T_\sigma\in p$. 
If $\tau < \kappa$
and $D_\sigma \prec D_\tau = (C,F)$, then
$t_\tau \in E_\tau = \bigcap_{r\in F}r^{-1}C \subseteq s^{-1}B$, so that $s \cdot p \in \closure{B}$.
Because $s$ and $B$ were arbitrary, this shows that $S \cdot p \sub \closure{B}$ for all $B \in \B$. By the continuity of $\rho_p$, this implies $\b S \cdot p \sub \closure{B}$ for all $B \in \B$. Thus
$$\textstyle L_p \sub \b S \cdot p \sub \bigcap \set{\closure{B}}{B \in \B} \sub \bigcap \U,$$
completing the proof of the lemma.
\end{proof}

\begin{theorem}\label{thm:non-normal}
Let $S$ be a very weakly cancellative semigroup with $|S|=\kappa$, and
assume there is a uniform, finite bound on $|\Fix(a)|$ for $a\in S$.
Then for every minimal ultrafilter $p$ on $S$, $S^*\setminus\{p\}$ is not normal.
\end{theorem}
\begin{proof}
Let $L$ and $R$ be the minimal left and right ideals respectively
with $p\in L\cap R$. Let $e$ be the identity of $L\cap R$. Let 
$C=\closure{p\cdot E(R)}$. We claim that $L\setminus \{p\}$ and
$C\setminus\{p\}$ cannot be separated by open sets in $S^*\setminus\{p\}$.  Suppose
instead that we have open subsets $U$ and $V$ of $S^*$ such that
$L\setminus\{p\}\subseteq U$, $C\setminus\{p\}\subseteq V$, and $U\cap V\subseteq \{p\}$. 
Let $D=\{s\in S : s \cdot p \neq p\}$ and observe that, by Lemma 3.2, $S\setminus D$ is finite.

For each $s \in D$, let $W_s = S^*\setminus s \cdot C$. Now fix $s \in D$. Because $\lambda_s$ is continuous, 
$W_s$ is open in $S^*$. We claim also that $p \in W_s$.
We know that $p \cdot E(R) \cdot e = \{p \cdot e\} = \{p\}$. Because
$\rho_e\big(p \cdot E(R)\big) = \{p\}$, and $\rho_e$ is continuous on all of $\b S$, we have
$$C \cdot e = \rho_e\big(\closure{p\cdot E(R)}\big) = \closure{\rho_e(p\cdot E(R))} = \{p\}.$$
If $p \in s \cdot C$, then
$p = p \cdot e \in s \cdot C \cdot e = \{s \cdot p\}$, so $p = s \cdot p$. This contradicts the assumption that $s \in D$, so we may conclude that
$W_s$ is a neighborhood of $p$. Hence $L \subseteq U \cup W_s$. Furthermore, because $s$ was an arbitrary element of $D$, $L\subseteq \bigcap_{s \in D}(U \cup W_s)$. 

By
Lemma~\ref{lem:Gdelta}, there is a minimal left ideal $L'$ of $\beta S$ such that
$L' \neq L$ and $L' \subseteq \bigcap_{s\in D}(U \cup W_s)$. Let 
$f$ be the identity of $L'\cap R$.  
Now $p\in L=S^*\cdot p$ so
$p=q\cdot p$ for some $q\in S^*$. Since $q\in S^*$, $D\in q$ and so $p\in\closure{D\cdot p}$.
Therefore $p \cdot f = \rho_f(p) \in \rho_f\big(\closure{D \cdot p}\big) = \closure{D \cdot p \cdot f}$.

We claim $L' \cap \bigcap_{s\in D}W_s \cap (D \cdot p \cdot f) = \emp$. Suppose
instead that $s\in D$ and $s \cdot p \cdot f \in L' \cap \bigcap_{t \in D}W_t$.
Then, in particular, $s \cdot p \cdot f\in W_s$. But $p \cdot f \in C$, so $s \cdot p \cdot f \in s \cdot C = S^* \setminus W_s$, a contradiction.

As $p \cdot f \in \closure{D \cdot p \cdot f}$ and $L' \cap \bigcap_{s \in D}W_s \cap (D \cdot p \cdot f) = \emp$,
we have that $L' \cap \bigcap_{s \in D}W_s$ is not a neighborhood of $p \cdot f$ in $L'$.

Now $p \cdot f \in C \setminus \{p\} \subseteq V$ and $V$ is open in $S^*$, 
so $L' \cap (L' \cup V)$ is a neighborhood of $p \cdot f$ in $L'$. Therefore $L' \cap (L' \cup V)$ cannot be contained in $L' \cap \bigcap_{s \in D}W_s$.
Pick $q \in L'\cap (L'\cup V) \setminus (L' \cap \bigcap_{s \in D}W_s)$, and pick $s \in D$ such that $q \notin W_s$.
Because $L' \subseteq U \cup \bigcap_{t \in D}W_t$, we must
have $q \in U \cap V$ and $q \notin W_s$. But $q \notin W_s$ implies $q \neq p$, so this shows that $U \cap V \not\sub \{p\}$, as desired.
\end{proof}

\begin{corollary}\label{cor:plusnotnormal} Let $p\in K(\beta\ben,+)$.
Then $\ben^*\setminus\{p\}$ is not normal.\end{corollary}

\begin{proof} $(\ben,+)$ is cancellative and for
$a\in\ben$, $\{x\in \ben:x+a=a\}=\emp$.\end{proof}

\begin{corollary}\label{cor:timesnotnormal} Let $p\in K(\beta\ben,\cdot )$.
Then $\ben^*\setminus\{p\}$ is not normal.\end{corollary}

\begin{proof} $(\ben,\cdot )$ is cancellative and for
$a\in\ben$, $\{x\in \ben:xa=a\}=\{1\}$.\end{proof}

$K(\beta\ben,+)\cap K(\beta\ben, \cdot)=\emp$ by \cite[Corollary 13.15]{H&S},
so the results of Corollaries \ref{cor:plusnotnormal} and \ref{cor:timesnotnormal} do not overlap.

Let us note that while Theorems \ref{thm:butterfly} and \ref{thm:non-normal} are stated for the space $S^*$, the conclusions also hold when $S^*$ is replaced with either $\b S$ or $U(S)$. For $\b S$, this can be deduced directly from Theorems \ref{thm:butterfly} and \ref{thm:non-normal} 
(using the fact that if $p$ is a butterfly/non-normality point of a closed subspace of $X$, then it is one in $X$ too). For $U(S)$, it follows from making a few trivial modifications to the proofs presented already.

We turn now to the last two results of this section, which give two curious topological properties of the spaces of the form $p \cdot E(R)$. These results are proved under the extra assumption that $S$ is countable.

Recall that $x$ is a \emph{$P$-point} of a space $X$ if every countable intersection of neighborhoods of $x$ is a neighborhood of $x$. $X$ is a \emph{$P$-space} if all its points are $P$-points or, equivalently, if countable intersections of open sets are open.

\begin{theorem}\label{thm:Pspace}
Suppose $S$ is countable, weakly cancellative, and has a uniform, finite bound on $|\Fix(a)|$ for $a\in S$. 
Let $R$ be a minimal right ideal of $S^*$ and let $r \in R$. Then every $q \in r \cdot E(R)$ is a 
$P$-point of $\closure{r \cdot E(R)}$. In particular, $r \cdot E(R)$ is a $P$-space.
\end{theorem}
\begin{proof}
Let $q \in r \cdot E(R)$, and let $U_1,U_2,U_3,\dots$ be open neighborhoods of $q$ in 
$C = \closure{r \cdot E(R)}$. Let $L$ denote the minimal left ideal of $S^*$ containing $q$.

Setting $K_n = C \setminus U_n$ for every $n \in \N$, we have
$$\textstyle L \cap \closure{\bigcup_{n \in \N}K_n} \ \sub \ L \cap C \, = \, \{q\}$$
by Theorem~\ref{thm:transversals}. Let $\set{d_n}{n \in \N}$ be a countable dense subset of $L$ that does not contain $q$.
(Such a set exists because $L$ is separable, as $L = \closure{S \cdot q}$, $\{x\in S:x\cdot q=q\}$ is finite by 
Lemma \ref{lem:muove}, and $L$ has no isolated points by Lemma~\ref{lem:butterfly}). 
Taking $A = \bigcup_{n \in \N} K_n$ and $B = \set{d_n}{n \in \N}$, $A$ and $B$ are $\s$-compact subsets of $\b\N$ such that $\closure{A} \cap B = \0$ and $\closure{B} \cap A = \0$. By Theorem 3.40 in \cite{H&S}, this implies $\closure{A} \cap \closure{B} = \0$. As $q \in \closure{B}$, we have $q \notin \closure{A}$. Taking complements, this implies $q$ is in the interior of $\bigcap_{n \in \N}U_n$. This shows $q$ is a $P$-point of $C$.
\end{proof}

\begin{corollary}\label{cor:Borel}
Suppose $S$ is countable, weakly cancellative, and has a uniform, finite bound on $|\Fix(a)|$ for $a\in S$. 
Let $R$ be a minimal right ideal of $S^*$. Then $E(R)$ is not Borel in $\b S$, nor is $q \cdot E(R)$ for any $q \in R$.
\end{corollary}
\begin{proof} Since $q\cdot E(R)=E(R)$ if $q\in E(R)$, it suffices to establish the second conclusion.
By (a very special case of) Lemma \ref{lem:Gdelta}, there are $2^{\continuum}$ minimal left ideals in $\beta S$ and $q\cdot E(R)$
meets each of them, so $|q\cdot E(R)|=2^{\continuum}$.
Theorem~\ref{thm:Pspace} implies that any compact subset of $q \cdot E(R)$ is finite (because every 
subspace of a $P$-space is a $P$-space, but infinite compact spaces are never $P$-spaces by \cite[Exercise 4K1]{GJ}).
Applying Lemma 3.1 from \cite{H&S2}, this implies that $q \cdot E(R)$ is not Borel in $\b S$.
(Lemma 3.1 of \cite{H&S2} says that any Borel subset of $\beta\ben$ is the union of at most
$\continuum$ compact sets, but the proof only uses the fact that $\ben$ is countable.)
\end{proof}

Once again, we note that weak cancellativity alone is not enough to prove 
Theorem~\ref{thm:Pspace} or Corollary~\ref{cor:Borel}. The semigroup $(\N,\vee)$ is weakly cancellative, 
but it has a single minimal right ideal, namely $\N^*$ itself, and $E(\ben^*)=\ben^*$. Clearly $\N^*$ is not a $P$-space, and it is Borel in $\b\N$.

\section{A negative result}\label{sec:boo}

In this final section, we address the natural question of whether, given a $\Theta$-maximal filter $\F$ 
and a $\Sigma$-maximal filter $\G$, their union $\F \cup \G$ must generate an ultrafilter. We show that
it is consistent with \zfc that the answer is negative.
More precisely, we will use the hypothesis  $\pseudo = \continuum$ (a weak form of Martin's Axiom) to 
construct a $\Theta$-maximal filter $\F$ on $\N$ and a $\Sigma$-maximal filter $\G$ on $\N$ such 
that $\F \cup \G$ does not generate an ultrafilter. 

In light of Lemma~\ref{lem:thetamaximal}, the assertion that some such $\F$ and $\G$ exist is equivalent to the assertion that there is a minimal left ideal $L$ and a $\Sigma$-maximal filter $\G$ such that $L \cap \widehat \G$ contains more than one point. In other words, it is equivalent to the assertion that not every $\Sigma$-maximal filter corresponds to a closed transversal for the minimal left ideals.

The hypothesis $\pseudo=\continuum$ is used indirectly in order to invoke a result from \cite{Brian} to prove Lemma~\ref{lem:ptheta} below. 
In order to keep this section relatively self-contained, we also include a (short) derivation of Lemma~\ref{lem:ptheta} from \ch. This latter hypothesis is stronger (i.e., \ch implies $\pseudo=\continuum$) so that proving the result from $\pseudo=\continuum$ is ``better'' in some sense. But either hypothesis is, of course, adequate to establish consistency with \zfc, and the reader who wishes to do so may ignore any further mention of $\pseudo$ and $\continuum$ and read this section as a self-contained proof carried out in $\zfc+\ch$.

Let us say that $A \sub \N$ is \emph{nicely thick} if it is thick and, for every minimal left ideal $L$, either $L \sub A^*$ or $L \cap A^* = \0$.

\begin{lemma}
Suppose $\set{I_n}{n \in \N}$ is a partition of $\N$ into intervals such that $\lim_{n \to \infty}\card{I_n} = \infty$. 
For every infinite $A \sub \N$, the set $\bigcup \set{I_n}{n \in A}$ is nicely thick.
\end{lemma}
\begin{proof} This is an immediate consequence of Theorem \ref{thm:npartition}
with $n=2$, $Z_1=A$, and $Z_2=\ben\setminus A$.
\end{proof}

\begin{lemma}\label{lem:nice}
For every thick set $A$, there is a nicely thick $B \sub A$. Furthermore, there is some such $B$ with the property that $A^* \setminus B^*$ contains a minimal left ideal.
\end{lemma}
\begin{proof}
Suppose $A$ is thick. By a routine recursion argument, we may pick a sequence $I_1, I_2, I_3, \dots$ of intervals such that
\begin{itemize}
\item $\set{I_n}{n \in \N}$ is a partition of $\N$,
\item $\lim_{n \to \infty}\card{I_n} = \infty$, and
\item $\bigcup \set{I_n}{n \text{ is even}} \sub A$.
\end{itemize}
By the previous lemma, $B = \bigcup \set{I_n}{n \text{ is even}}$ is a nicely thick subset of $A$. For the second assertion, take $B = \bigcup \set{I_n}{n \text{ is a multiple of }4}$ instead. 
Then $A \setminus B$ contains the (nicely) thick set $\bigcup \set{I_n}{n \text{ is an odd multiple of }2}$, which implies that $A^* \setminus B^* = (A \setminus B)^*$ contains a minimal left ideal by Lemma~\ref{lem:BHM}.
\end{proof}

Let us say that $X \sub \N^*$ is \emph{left-separating} if for every minimal left ideal $L$, either $L \sub X$ or $L \cap X = \0$. Thus, if $A \sub \N$ then (by definition) $A^*$ is left-separating if and only if $A$ is nicely thick.

\begin{lemma}\label{lem:separating}
The collection of all left-separating subsets of $\N^*$ is closed under arbitrary unions and intersections, and under taking relative complements. 
\end{lemma}
\begin{proof}
Let $\mathcal X$ be a collection of left-separating subsets of $\N^*$, and let $L$ be a minimal left ideal. Either $(1)$ some $X \in \mathcal X$ contains $L$, in which case $\bigcup \mathcal X$ contains $L$, or $(2)$ no $X \in \mathcal X$ contains $L$, in which case $L \cap \bigcup \mathcal X = \0$. As $L$ was arbitrary, $\bigcup \mathcal X$ is left-separating. A similar argument shows the collection of left-separating subsets of $\N^*$ is closed under arbitrary intersections and taking relative complements.
\end{proof}

\begin{lemma}\label{lem:ptheta}
Assume $\pseudo = \continuum$ (or \ch). If $\a$ is an ordinal with $\a < \continuum$ and $\seq{A_\b}{\b < \a}$ is a 
sequence of thick subsets of $\ben$, well-ordered in type $\a$, such that $A_\b^* \supseteq A_\g^*$ whenever $\b < \g$, then there is a thick set $A_\a$ such that $A_\b^* \supseteq A_\a^*$ for all $\b < \a$.
\end{lemma}
\begin{proof}
This follows from \cite[Theorem 3.4]{Brian}.

More precisely, in \cite{Brian} the cardinal number $\tower_\Theta$ is defined to be the least 
cardinal $\k$ such that the conclusion of the present lemma is true for all $\a < \k$. 
Thus the present lemma can be rephrased as follows: \emph{if $\pseudo = \continuum$, then $\tower_\Theta = \continuum$.} 
But Theorem 3.4 in \cite{Brian} asserts that $\tower_\Theta = \tower$, and it is known that 
$\pseudo \leq \tower \leq \continuum$. Hence $\pseudo = \continuum$ implies 
$\tower_\Theta = \continuum$, as claimed.
\end{proof}

\begin{proof}[Proof of Lemma~\ref{lem:ptheta} from \ch]
If $\alpha=\delta+1$, let $A_\alpha=A_\delta$.  So assume $\alpha$
is a (nonzero) limit ordinal. We claim that for each $F\in\pf(\alpha)$,
$\bigcap_{\delta\in F}A_\delta$ is thick. For such $F$, let $\gamma=\max F$.
Then $A_\gamma^*\subseteq \bigcap_{\delta\in F}A_\delta^*=(\bigcap_{\delta\in F}A_\delta)^*$
so $G=A_\gamma\setminus\bigcap_{\delta\in F}A_\delta$ is finite. Since
$A_\gamma$ is thick, by Lemma \ref{lem:thickminus}, $A_\gamma\setminus G$ is thick
and $A_\gamma\setminus G\subseteq\bigcap_{\delta\in F}A_\delta$.

Now $\a$ is countable, by \ch. Thus we may enumerate $\{A_\delta:\delta<\alpha\}$ as $\langle B_n\rangle_{n=1}^\infty$.
For each $n$, let $C_n=\bigcap_{t=1}^nB_n$.  Then each $C_n$ is thick.
For $n\in\ben$, pick $x_n\in\ben$ such that $\{x_n+1,x_n+2,\ldots,x_n+n\}\subseteq C_n$
and let $A_\alpha=\bigcup_{n=1}^\infty \{x_n+1,x_n+2,\ldots,x_n+n\}$.
Then $A_\alpha$ is thick. Given $\delta<\alpha$, pick $n\in\ben$ such that $A_\delta=B_n$.
Then $A_\alpha\setminus A_\delta\subseteq \bigcup_{t=1}^{n-1}\{x_t+1,x_t+2,\ldots,x_t+t\}$,
so $A_\alpha^*\subseteq A_\delta^*$.
\end{proof}

\begin{lemma}\label{lem:2sets}
Assuming $\pseudo = \continuum$ (or \ch), there exist disjoint open subsets $U$ and $V$ of $\N^*$ such that
\begin{itemize}
\item $U$ and $V$ are left-separating.
\item there is a minimal left ideal $L$ such that
\begin{itemize}
\item[$\circ$] $L$ has a neighborhood basis of left-separating clopen sets; i.e., for every open set $W$ containing $L$, there is a nicely thick $A \sub \N$ with $L \sub A^* \sub W$.
\item[$\circ$] $L \sub \closure{U \cap K(\b\N)}$.
\item[$\circ$] $L \sub \closure{V \cap K(\b\N)}$.
\end{itemize}
\end{itemize}
\end{lemma}
\begin{proof}
We begin the proof with the construction of a basis for a $\Theta$-maximal filter $\F$. Ultimately, $U$ and $V$ will be defined from our basis for $\F$, and $\widehat \F$ will be the minimal left ideal $L$ mentioned in the statement of the lemma.

To construct $\F$, fix an enumeration $\set{A_\a}{\a < \continuum}$ of $\Theta$. Using transfinite recursion, we will construct a well-ordered sequence $\seq{X_\a}{\a < \continuum}$ of nicely thick sets such that 
\begin{itemize}
\item $X_\a^* \supseteq X_\b^*$ whenever $\a < \b$,
\item $X_\a^* \setminus X_\b^*$ contains a minimal left ideal whenever $\a < \b$, and 
\item for every $\a < \continuum$, either $X_{\a+1}^* \sub A_\a^*$ or $X_{\a+1} \cap A_\a \notin \Theta$.
\end{itemize}
For the base stage of the recursion, set $X_0 = \N$. 

At the successor stage $\a+1$ of the recursion, assuming $X_\a$ has already been defined, there are two cases. If $X_\a \cap A_\a \notin \Theta$, then choose $X_{\a+1}$ to be any nicely thick subset of $X_\a$ such that $X_\a^* \setminus X_{\a+1}^*$ 
contains a minimal left ideal. (This is possible by Lemma~\ref{lem:nice}.) If $X_\a \cap A_\a \in \Theta$, then let $X_{\a+1}$ be some nicely thick set contained in $X_\a \cap A_\a$ with the property that $X_\a^* \setminus X_{\a+1}^*$ 
contains a minimal left ideal. (Again, this is possible by Lemma~\ref{lem:nice}.)

If $\a$ is a limit ordinal with $\a < \continuum$, then at stage $\a$ of the recursion we will have a sequence $\seq{X_\b}{\b < \a}$ of nicely thick sets such that $X_\b^* \supseteq X_\g^*$ whenever $\b < \g$. 
By Lemma~\ref{lem:ptheta}, there is a thick set $X_\a^0$ such that $X_\b^* \supseteq (X_\a^0)^*$ for all $\b < \a$. Let $X_\a$ be any nicely 
thick set contained in $X_\a^0$. (One exists by Lemma~\ref{lem:nice}.) This completes the recursion.

It is clear that $\set{X_\a}{\a < \continuum}$ is a filter base.
Let $\F$ be the filter generated by this base, and let 
$$\textstyle L = \widehat \F = \bigcap_{\a < \continuum}X_\a^*.$$
For each $\a < \continuum$, let $R_\a = X_\a^* \setminus X_{\a+1}^*$. Let
$$\textstyle U = \bigcup \set{R_\a}{\a < \continuum,\, \a \text{ even}}\hbox{ and let } 
V = \bigcup \set{R_\a}{\a < \continuum,\, \a \text{ odd}}$$
where, as usual, an ordinal is called even (respectively, odd) if it is equal to $\lambda+n$, where $\lambda$ is a limit ordinal and $n$ is even (respectively, odd). 
Thus $U$ and $V$ are formed each as the union of alternating clopen ``rings'' (the $R_\a$) from the nested sequence $X_0^* \supsetneq X_1^* \supsetneq X_2^* 
\supsetneq \dots \supsetneq X_\a^* \supsetneq X_{\a+1}^* \supsetneq \dots$ (where we ignore the non-clopen ``rings'' occurring at limit stages).

It is clear that $U$ and $V$ are disjoint open sets. That $U$ and $V$ are left separating follows from their definition and  Lemma~\ref{lem:separating}.

\begin{claim}
$L$ is a minimal left ideal.
\end{claim}
\begin{proof}[Proof of claim]
By Lemma~\ref{lem:thetamaximal}, it suffices to show that $\F$ is a $\Theta$-maximal filter.
If $A \in \Theta$, then $A = A_\a$ for some $\a < \continuum$. At stage $\a+1$ of our recursion, we ensured that either $X_{\a+1} \cap A_\a \notin \Theta$ or that $X_{\a+1} \sub A_\a$. Thus $A \in \Theta$ implies that either $A \in \F$ or $A \cap X \notin \Theta$ for some $X \in \F$. Thus there is no proper extension of $\F$ containing only thick sets.
\end{proof}

\begin{claim}
For every open $W \supseteq L$, there is a nicely thick $A \sub \N$ with $L \sub A^* \sub W$. In fact, there is some $\a < \continuum$ such that $L \sub X_\b^* \sub W$ for all $\a \leq \b < \continuum$.
\end{claim}
\begin{proof}[Proof of claim]
If $W$ is open and $W \supseteq L = \bigcap_{\a < \continuum}X_\a^*$, 
then, because this is a decreasing intersection of compact sets, there is some $\a < \continuum$ such that 
$L \sub X_\b^* \sub W$ for all $\b \geq \a$. (Otherwise, $\{X_\a^*\setminus W:\a < \continuum\}$
would be a set of closed sets with the finite intersection property, so would have
nonempty intersection.) Each $X_\b$ is nicely thick, so this proves the claim.
\end{proof}

\begin{claim}
$L \sub \closure{U \cap K(\b\N)}$, and
$L \sub \closure{V \cap K(\b\N)}$.\end{claim}
\begin{proof}[Proof of claim]
We will prove that $L \sub \closure{U \cap K(\b\N)}$ only, as the corresponding assertion for $V$ is proved in the same way.

Suppose $p \in L$ and let $W$ be a neighborhood of $p$. We must show that $W \cap U \cap K(\b\N) \neq \0$. 
By a previous claim, $L$ is a minimal left ideal. By Lemma \ref{lem:sigmamaximal} and Theorem \ref{thm:basis}, 
there are $A,B \sub \N$ such that $L \sub A^*$, $B$ is syndetic, and $p \in A^* \cap B^* \sub W$.

By the previous claim, there is some $\a < \continuum$ such that $X_\b^* \sub A^*$ for all $\b \geq \a$. Let $\b \geq \a$ be an even successor ordinal, so that 
$$R_\b = X_\b^* \setminus X_{\b+1}^* \sub A^* \cap U.$$

By construction, $R_\b$ contains a minimal left ideal $L'$. Then 
$$L' \sub R_\b \cap K(\b\N) \sub A^* \cap U \cap K(\b\N).$$
But $B$ is syndetic, which implies $B^* \cap L' \neq \0$ by Lemma~\ref{lem:BHM}. Hence
$$\0 \neq B^* \cap L' \sub B^* \cap A^* \cap U \cap K(\b\N) \sub W \cap U \cap K(\b\N).$$
This shows $W \cap U \cap K(\b\N) \neq \0$, as desired.
\end{proof}

These claims complete the proof of the lemma. 
\end{proof}

\begin{theorem}
Assuming $\pseudo = \continuum$ (or \ch), there is a $\Sigma$-maximal filter $\G$ on $\N$ and a $\Theta$-maximal filter $\F$ on $\N$ such that $\G \cup \F$ does not generate an ultrafilter.
\end{theorem}
\begin{proof}
Let $U$, $V$, and $L$ be as described in Lemma~\ref{lem:2sets}. Let $\F = \bigcap{L}$.

Let $R$ and $R'$ be two different minimal right ideals of $\N^*$, and define
$$C_0 = \big(U \cap \closure{E(R)}\big) \cup \big(V \cap \closure{E(R')}\big) \cup (\N^* \setminus (U \cup V)).$$
If $L'$ is a minimal left ideal, then because $U$ and $V$ are both left-separating, there are three possibilities:
\begin{enumerate}
\item $L' \sub U$, in which case $C_0 \cap L' = \closure{E(R)} \cap L' \neq \0$,
\item $L' \sub V$, in which case $C_0 \cap L' = \closure{E(R')} \cap L' \neq \0$, or
\item $L' \sub \N \setminus (U \cup V)$, in which case $L' \sub C_0$.
\end{enumerate} 
In any case, $C_0 \cap L' \neq \0$ for every minimal left ideal $L'$. 

Let $\G_0 = \bigcap C_0$, and observe that $\G_0 \sub \Sigma$ by the previous paragraph and Lemma~\ref{lem:BHMagain}. 
Using Zorn's Lemma, extend $\G_0$ to a $\Sigma$-maximal filter $\G$ and let $C=\widehat{\G}$.

We claim that $\F \cup \G$ does not generate an ultrafilter. We will prove the equivalent assertion that $\widehat \F \cap \widehat \G = L \cap C$ contains at least two points. 

Let $e$ denote the unique point of $L \cap E(R)$, and let $e'$ denote the unique point of $L \cap E(R')$. Because $R \neq R'$, $e \neq e'$, and we claim that $e,e' \in L \cap C$. We will prove only that $e \in L \cap C$, because the proof for $e'$ is the same. We know $e \in L$ already, so we must show $e \in C$.

Aiming for a contradiction, suppose $e \notin C$. Then there is some open $W$ such that $e \in W$ and $W \cap C = \0$. 
By Theorem~\ref{thm:basis}, there are $A,B \sub \N$ such that $L \sub A^*$ and $E(R) \sub B^*$ and $e \in A^* \cap B^* \sub W$. By our choice of $L$, we may (and do) assume that $A$ is nicely thick. By our choice of $L$ and $U$, we have $e \in \closure{U \cap K(\b\N)}$, so
$$A^* \cap B^* \cap U \cap K(\b\N) \neq \0.$$ 
Let $p \in A^* \cap B^* \cap U \cap K(\b\N)$, and let $L'$ be the minimal left ideal containing $p$. Observe that $L \cap U = \0$, because $L \sub \closure{V}$ and $U$ and $V$ are disjoint open sets. As $p \in U$ and $p \in L'$, this implies $L' \neq L$. Because $U$ and $A^*$ are both left-separating and $p \in A^* \cap U$, we have
$$L' \sub A^* \cap U \cap K(\b\N)$$
Recalling that $E(R) \sub B^*$, this shows that
$$L' \cap E(R) \sub A^* \cap B^* \cap U \cap K(\b\N).$$
Let $f$ denote the unique element of $L' \cap E(R)$. On the one hand, we just showed that $f \in A^* \cap B^* \sub W$, which implies $f \notin C$. On the other hand, $f \in U$ and $L' \sub U$, which implies that 
$$C_0 \cap L' = \closure{E(R)} \cap L' = \{f\}\,.$$
(The second equality comes from Theorem~\ref{thm:transversals}.) This shows that $C \cap L' = \0$. But $C = \widehat \G$ with $\G \sub \Sigma$, so this contradicts Lemma~\ref{lem:BHMagain}.
\end{proof}



\end{document}